\documentclass[11pt]{amsart}
\usepackage[T1]{fontenc}
\usepackage{latexsym,amssymb,amsmath,mathrsfs}

\usepackage{amsthm}
\usepackage{longtable}
\usepackage{epsfig}
\usepackage{hhline}
\usepackage{epic}

\catcode`\Ž=13
\defŽ{\'e}
\catcode`\ˆ=13
\defˆ{\`a}
\catcode`\=13
\def{\`e}
\catcode`\‰=13
\def‰{\^a}
\catcode`\=13
\def{\c c}
\catcode`\™=13
\def™{\^o}
\catcode`\š=13
\defš{\'E}
\catcode`\Š=13
\defŠ{\`A}
\catcode`\Ÿ=13
\defŸ{\^o}
\catcode`\=13
\def{\^e}
\catcode`\"=13
\def"{\^\i}
\catcode`\=13
\def{\`u}
\catcode`\•=13
\def•{\"\i}

   \newcommand{\sym}{\mathfrak{S}}

   \newcommand{\N}{\mathbb{N}}

   \newcommand{\la}{\lambda}

  \newcommand{\ba}{\beta}
  \newcommand{\sa}{\sigma}

\newcommand{\cal}[1]{\mathcal{#1}}

  \newcommand{\modd}{ \; (\mbox{mod} \; d)}
 
  \newcommand{\bd}{\bar{d}}
  \newcommand{\lad}{\bar{c}_d(\la)}

\newcommand{\ladq}{\la^{(\bar{d})}}
\newcommand{\bdqla}{\bar{q}_d(\la)}
 \newcommand{\Bla}{ {\cal B}(\la)}
 \newcommand{\Blad}{ {\cal B}(\bar{c}_d(\la))}
 \newcommand{\Bladq}{ {\cal B}(\bdqla)}
 \newcommand{\Dla}{D(\la)}
 \newcommand{\Dlad}{D(\lad)}
 \newcommand{\Dbdqla}{D(\bdqla)}
 \newcommand{\Dmu}{D(\mu)} 
  
  \newcommand{\floor}[1]{\lfloor #1 \rfloor}
    
  \newcommand{\cH}{ {\cal H}}   
 \newcommand{\cB}{ {\cal B}}
 \newcommand{\cP}{ {\cal P}}
 \newcommand{\cNB}{ {\cal{NB}}}
 \newcommand{\cDP}{ {\cal{DP}}}

 \newcommand{\cHij}{ {\cal H}_{i \rightarrow j}}
 \newcommand{\cBij}{ {\cal B}_{i \rightarrow j}}
 
 \newcommand{\cNBij}{ {\cal{NB}}_{i \rightarrow j}}

  \newcommand{\cHji}{ {\cal H}_{j \rightarrow i}} 
  \newcommand{\cBji}{ {\cal B}_{j \rightarrow i}}
 
 \newcommand{\cNBji}{ {\cal{NB}}_{j \rightarrow i}}

  \newcommand{\cBijstar}{ {\cal B}_{i^* \rightarrow j^*}}
 \newcommand{\cBjistar}{ {\cal B}_{j^* \rightarrow i^*}}

 \newcommand{\ov}{ \overline}   
 \newcommand{\sij}{ \{ij\}}   
  \newcommand{\sijs}{ \{i^*j^*\}}

\long\def\symbolfootnote[#1]#2{\begingroup\def\thefootnote{\fnsymbol{footnote}}
\footnote[#1]{#2}\endgroup}

\newtheorem{theorem}{Theorem}[section] 
\newtheorem{lemma}[theorem]{Lemma}     
\newtheorem{corollary}[theorem]{Corollary}

\title[On bar lengths in partitions]{On bar lengths in partitions}

\author{Jean-Baptiste Gramain}
\address{Institut de Mathématiques de Jussieu, 
Université Denis Diderot, Paris VII \\
UFR De Mathématiques, 2 Place Jussieu, F-75251 Paris Cedex 05, France
}
\email{gramain@math.jussieu.fr}

\author{J\o rn B. Olsson}
\address{
Department of Mathematical Sciences, University of Copenhagen\\
Universitetsparken 5,DK-2100 Copenhagen \O, Denmark
}
\email{olsson@math.ku.dk}

\begin{document}
\maketitle

\begin{abstract}
In this paper, we present, given a odd integer $d$, a decomposition of the multiset of bar lengths 
of a bar partition $\la$ as the union of two multisets, one consisting
of the bar lengths in its $\bar{d}$-core partition $\lad$ 
and the other consisting of modified bar lengths in its $\bar{d}$-quotient partition. 
In particular, we obtain that the multiset of bar lengths in
$\lad$ is a sub-multiset of the multiset of bar lengths in
$\la$. Also, we obtain a relative bar formula for the degrees of spin
characters of the Schur extensions of $\sym_n$. 
The proof involves a recent similar result for partitions, 
proved in \cite{BeGrOl}. 

\end{abstract}

\symbolfootnote[0]{2000 Mathematics Subject Classification 20C30 (primary), 20C15, 20C20 (secondary)}

\symbolfootnote[0]{Keywords: Partitions, Bar Partitions, Bar Lengths, Symmetric Group, Covering Groups}

\section{Introduction}\label{intro}
For any positive integer $n$, we call any partition $\la$ of $n$ into
distinct parts a {\emph{bar partition}} of $n.$  
It was proved by I. Schur (in \cite{Schur}) that the bar partitions of
$n$ canonically label the associate classes of irreducible projective 
representations of the symmetric group $\sym_n$, or the associate
classes of faithful irreducible characters 
({\emph{spin characters}}) of a 2-fold covering group $\widehat{\sym_n}$ of $\sym_n$.

In \cite[Theorem 1]{Morris}, A.O. Morris proved a formula 
(the {\it bar formula}) for the degrees of the spin characters 
analogous to the celebrated hook formula (\cite[Theorem 2.3.21]{James-Kerber}) 
for the irreducible characters of $\sym_n.$ The bar
formula is a reformulation of the original degree formula proved by
Schur in \cite[IX, p.235]{Schur}. We state the bar formula below. 
In the bar formula the role played by hooks and hook lengths
of partitions is replaced by {\emph{bars}} and {\emph{bar lengths}} of bar partitions. 

If $\la=(a_1 > \cdots > a_m >0)$ is a bar partition of $n$, then the multiset of bar lengths in 
$\la$ is
$${\cal B}(\la)=\displaystyle \bigcup_{1 \leq i \leq m} 
\{ 1, \, \ldots , \, a_i \} \cup \{ a_i + a_j \, | \, j>i\} \setminus \{ a_i - a_j \, | \, j >i \} .$$
Writing $\pi {\cal B}(\la)$ for the product of all the bar lengths in
$\la$, we then have the bar formula 
for the degree of  a spin character $\rho_{\la}$ of $\widehat{\sym_n}$ labelled by $\la:$  
$$\rho_{\la}(1)= 2^{\floor{(n-m)/2}} \frac{n!}{\pi {\cal B}(\la)},$$
where, for any rational number $x$, $\floor{x}$ denotes the integral part of $x$.

\smallskip
For any odd integer $d\geq 3$, it is well-known that the bar partition
$\la$ is uniquely determined by its {\emph{$\bd$-core}} $\lad$ 
and its {\emph{$\bd$-quotient}} $\ladq$ (see e.g. \cite[Proposition
4.2]{Olsson-Combinatorics}). 
The $\bd$-core partition $\lad$ of $\la$ is 
obtained by removing from $\la$ all the bars of length divisible by
$d$, while $\ladq$ encodes the information about these bars. 

\smallskip
For any bar partition $\la$ of $n$ and odd integer $d \geq 3$, we
denote by $\bdqla$ the (unique) bar partition which has an {\it empty}
$\bd$-core and the same $\bd$-quotient as $\la$. We refer to $\bdqla$
as the $\bd$-{\it quotient partition} of $\la$ and have that 
$|\la|=|\lad|+|\bdqla|$ 
(see \cite[Corollary 4.4]{Olsson-Combinatorics}). This identity is
reflected in our main result on the decomposition of the multiset of
bar lengths  
(Theorem \ref{main}). It states that the multiset $\Bla$ of bar lengths 
in $\la$ is the union of $\Blad$ and  $\widetilde{\cal B}(\bdqla)$  
where the multiset $ \widetilde{\cal B}(\bdqla)$ is obtained from
$\Bladq$ by modifying its elements in an explicitly controlled way, 
depending on the $\bd$-core of $\la.$
As an immediate corollary we obtain that $\Blad$ is contained in $\Bla$.

\smallskip
In Section 2, we describe the {\emph{doubling}} of bar partitions; 
this construction was first suggested by I. G. Macdonald in
\cite{Macdonald}, and then studied by A. O. Morris and A. K. Yaseen in
\cite{Morris-Yaseen}. It allows us to see all the bar lengths in a bar
partition as hook lengths in a larger partition. We present the
construction, as well as interpretations of the bar core and bar
quotient in this setting. 
In Section 3, we introduce a number of subsets of the set of hooks in
the doubled partition, 
and derive from Macdonald's construction a number of properties of
hook lengths and bar lengths. In Section 4, we then apply the results
of \cite{BeGrOl} to deduce our main result, Theorem \ref{main}. We
then finally apply the theorem to give a $d$-version of the bar 
formula (a relative bar formula) .

\section{The Macdonald construction}\label{doubling}
Let $n \geq 1$ be any integer, and $\la=(a_1 > \cdots > a_m >0)$ be a 
bar partition of $n$. I. G. Macdonald presented in 
\cite[Chapter III, p. 135]{Macdonald} a construction for the 
{\emph{doubling}} of $\la$, which we present here using the example 
given by the bar partition $\la=(7, \, 5, \, 3, \, 2)$ of $n=17$.

The {\emph{shifted Young diagram}} $S(\la)$ of $\la$ is obtained from 
the usual Young diagram of $\la$ by moving, for each $i \geq 1$, the 
$i$th row $(i-1)$ squares to the right. In our example, this gives

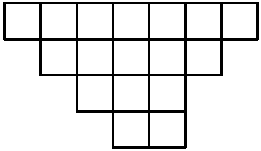
\setlength{\unitlength}{1.2mm}

\linethickness{0.2mm}

\begin{center}

\begin{picture}(30,16)

\put(16,0){\line(1,0){8}}
\put(12,4){\line(1,0){12}}
\put(8,8){\line(1,0){20}}
\put(4,12){\line(1,0){28}}
\put(4,16){\line(1,0){28}}
\put(4,16){\line(0,-1){4}}
\put(8,16){\line(0,-1){8}}
\put(12,16){\line(0,-1){12}}
\put(16,16){\line(0,-1){16}}
\put(20,16){\line(0,-1){16}}
\put(24,16){\line(0,-1){16}}
\put(28,16){\line(0,-1){8}}
\put(32,16){\line(0,-1){4}}

\end{picture}
\end{center}
Equivalently, $S(\la)$ can be seen as the part above the diagonal in the Young diagram 
of the {\emph{doubled partition}} 
$D(\la)=(a_1, \, \ldots a_m \, | \, a_1-1, \, \ldots , \, a_m-1)$ of $2n$ 
(given in the Frobenius' notation, see e.g. \cite[Chapter
I]{Macdonald}). 
In our example, we obtain the partition 
$D(\la)=(8, \, 7, \, 6, \, 6, \, 4, \, 2, \, 1)$ of $2n=34$, which has Young diagram
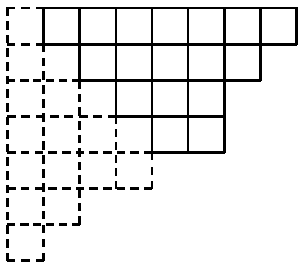
\setlength{\unitlength}{1.2mm}

\linethickness{0.2mm}

\begin{center}

\begin{picture}(30,29)

\put(16,12){\line(1,0){8}}
\put(12,16){\line(1,0){12}}
\put(8,20){\line(1,0){20}}
\put(4,24){\line(1,0){28}}
\put(4,28){\line(1,0){28}}
\put(4,28){\line(0,-1){4}}
\put(8,28){\line(0,-1){8}}
\put(12,28){\line(0,-1){12}}
\put(16,28){\line(0,-1){16}}
\put(20,28){\line(0,-1){16}}
\put(24,28){\line(0,-1){16}}
\put(28,28){\line(0,-1){8}}
\put(32,28){\line(0,-1){4}}

\dashline[+60]{1}(0,0)(4,0)
\dashline[+60]{1}(0,4)(8,4)
\dashline[+60]{1}(0,8)(16,8)
\dashline[+60]{1}(0,12)(16,12)
\dashline[+60]{1}(0,16)(12,16)
\dashline[+60]{1}(0,20)(8,20)
\dashline[+60]{1}(0,24)(4,24)
\dashline[+60]{1}(0,28)(4,28)
\dashline[+60]{1}(0,28)(0,0)
\dashline[+60]{1}(4,24)(4,0)
\dashline[+60]{1}(8,20)(8,4)
\dashline[+60]{1}(12,16)(12,8)
\dashline[+60]{1}(16,12)(16,8)

\end{picture}

\end{center}
Filling the boxes of the Young diagram of $D(\la)$ with the 
corresponding hook lengths, we obtain that the bar lengths 
in $\la$ are those hook lengths that appear in the subdiagram $S(\la)$. 
In our example, we get 
$$\Bla=\{ 12, \, 10, \, 9, \, 8, \, 7, \, 7, \, 6, \, 5, \, 5, \, 4, \, 3, \, 3, \, 2, \, 2, \, 1, \, 1, \, 1\}:$$
\setlength{\unitlength}{1.2mm}

\linethickness{0.2mm}

\begin{center}

\begin{picture}(30,27)

\put(16,12){\line(1,0){8}}
\put(12,16){\line(1,0){12}}
\put(8,20){\line(1,0){20}}
\put(4,24){\line(1,0){28}}
\put(4,28){\line(1,0){28}}
\put(4,28){\line(0,-1){4}}
\put(8,28){\line(0,-1){8}}
\put(12,28){\line(0,-1){12}}
\put(16,28){\line(0,-1){16}}
\put(20,28){\line(0,-1){16}}
\put(24,28){\line(0,-1){16}}
\put(28,28){\line(0,-1){8}}
\put(32,28){\line(0,-1){4}}

\dashline[+60]{1}(0,0)(4,0)
\dashline[+60]{1}(0,4)(8,4)
\dashline[+60]{1}(0,8)(16,8)
\dashline[+60]{1}(0,12)(16,12)
\dashline[+60]{1}(0,16)(12,16)
\dashline[+60]{1}(0,20)(8,20)
\dashline[+60]{1}(0,24)(4,24)
\dashline[+60]{1}(0,28)(4,28)
\dashline[+60]{1}(0,28)(0,0)
\dashline[+60]{1}(4,24)(4,0)
\dashline[+60]{1}(8,20)(8,4)
\dashline[+60]{1}(12,16)(12,8)
\dashline[+60]{1}(16,12)(16,8)

\put(1,1){\bf{$1$}}
\put(1,5){\bf{$3$}}
\put(1,9){\bf{$6$}}
\put(1,13){\bf{$9$}}
\put(0.5,17){\bf{$10$}}
\put(0.5,21){\bf{$12$}}
\put(0.5,25){\bf{$14$}}
\put(5,5){\bf{$1$}}
\put(5,9){\bf{$4$}}
\put(5,13){\bf{$7$}}
\put(5,17){\bf{$8$}}
\put(4.5,21){\bf{$10$}}
\put(4.5,25){\bf{12}}
\put(9,9){\bf{$2$}}
\put(9,13){\bf{$5$}}
\put(9,17){\bf{$6$}}
\put(9,21){\bf{8}}
\put(8.5,25){\bf{10}}
\put(13,9){\bf{$1$}}
\put(13,13){\bf{$4$}}
\put(13,17){\bf{5}}
\put(13,21){\bf{7}}
\put(13,25){\bf{9}}
\put(17,13){\bf{2}}
\put(17,17){\bf{3}}
\put(17,21){\bf{5}}
\put(17,25){\bf{7}}
\put(21,13){\bf{1}}
\put(21,17){\bf{2}}
\put(21,21){\bf{4}}
\put(21,25){\bf{6}}
\put(25,21){\bf{1}}
\put(25,25){\bf{3}}
\put(29,25){\bf{1}}

\end{picture}
\end{center}

We refer to \cite[Section 2.7]{James-Kerber} or \cite[Section 
1]{Olsson-Combinatorics} for the basic facts about $\beta$-sets for
partitions and their relation to hooks. In particular, if $X$ is a
$\beta$-set for the partition $\la,$ then there is a canonical
correspondence between the hooks $z$ in $\la$ and pairs $(a,b)$ of
non-negative integers, where $a \in X, ~b \notin X$ and $a>b.$ The
length $h(z)$ of the hook $z$ is then $a-b.$
   
Now take any odd integer $d \geq 3$. We represent a
{\emph{$d$-normalized}} $\beta$-set $X$ 
for $\Dla$ (i.e. $|X|$ is a multiple of $d$) 
by placing beads on an abacus with $d$ runners. If $d=3$,
then in our example we can take 
$X=\{0, \,  1,  \, 3, \, 5,  \, 8, \, 11, \, 12, \, 14, \, 16 \}$, and we obtain

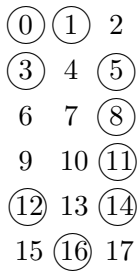
\setlength{\unitlength}{1mm}
\linethickness{0.2mm}
\begin{center}
\begin{picture}(30,37)

\put(0.5,1){$15$}
\put(0.5,7){$12$}
\put(1,13){$9$}
\put(1,19){$6$}
\put(1,25){$3$}
\put(1,31){$0$}
\put(6.5,1){$16$}
\put(6.5,7){$13$}
\put(6.5,13){$10$}
\put(7,19){$7$}
\put(7,25){$4$}
\put(7,31){$1$}
\put(12.5,1){$17$}
\put(12.5,7){$14$}
\put(12.5,13){$11$}
\put(13,19){$8$}
\put(13,25){$5$}
\put(13,31){$2$}

\put(2.2,8){\circle{5}}
\put(2,26.2){\circle{5}}
\put(2,32.2){\circle{5}}
\put(8.2,2){\circle{5}}
\put(8,32.2){\circle{5}}
\put(14.2,8){\circle{5}}
\put(14.2,14){\circle{5}}
\put(14,20){\circle{5}}
\put(14,26){\circle{5}}

\end{picture}
\end{center}

For any integer $\ell$, we denote by $[\ell]_d$ the
{\emph{$d$-residue}} of $\ell$, i.e. the least non-negative integer
congruent to $\ell \modd$. We label each node in the Young diagram 
of $\Dla$ by a $d$-residue as follows: the $d$-residue labelling the 
$(i, \, j)$-node is $[j-i]_d$. In particular, note that the
{\emph{diagonal nodes}} (which correspond to hooks whose lengths are
twice the size of the parts of $\la$) all have residue 0. 
Writing $H(\Dla)$ for the set of hooks in $\Dla$, we define, 
for each $0 \leq i, \, j \leq d-1$, the subset $H_{i \rightarrow
  j}(\Dla)$ of hooks of $\Dla$ whose hand node and foot node have 
$d$-residues $i$ and $[j+1]_d$ respectively. For any hook $z \in
H(\Dla)$, we have that $z \in H_{i \rightarrow j}(\Dla)$ if and only
if, in the abacus, $z$ corresponds to a bead $a$ on the $i$th runner 
and an empty spot $b$ on the $j$-th runner.  In this case, the length
$h(z)$ of $z$ satisfies $h(z) \equiv j-i \modd$ (see \cite[Section 2.7]{James-Kerber}
for details). For each $0 \leq i \neq j \leq d-1$, 
we write $H_{\{i\}}(\Dla)$ for $H_{i \rightarrow i}(\Dla)$  
and $H_{\{ij\}}(\Dla)$ for 
$H_{i \rightarrow j}(\Dla) \cup H_{j \rightarrow i}(\Dla)$. 
In particular, $\bigcup_{0 \leq i \leq d-1} H_{\{i\}}(\Dla)$ is 
the set of hooks of length divisible by $d$ in $\Dla$.

\smallskip

For a given runner in the abacus for $\Dla$ we may regard the positions
of the elements of $X$ as beads on the runner as a $\beta$-set. We obtain the
$d$-quotient $\Dla^{(d)}$ of $\Dla$ as the $d$-tuple of these
$\beta$-sets. The fact that we took a normalized $\beta$-set ensures
that the $d$-quotient we obtain is the same as the one we would obtain
by considering the {\emph{$d$ star diagram}} of $\Dla$ (\cite[Theorem
2.7.37]{James-Kerber}).
We may then reformulate \cite[Theorem 4]{Morris-Yaseen} as

\begin{theorem}
With the above notation, the $d$-quotient $\Dla^{(d)}$ of $\Dla$ has the form
$$\Dla^{(d)}= ( D(\mu_0), \, \mu_1, \, \ldots , \,
\mu_{\frac{d-1}{2}}, \, \mu_{\frac{d-1}{2}}^*, \, \ldots , \, \mu_1^*
),$$where $\mu_0$ is a bar partition, 
$\mu_1, \, \ldots , \, \mu_{\frac{d-1}{2}}$ 
are partitions, and $^*$ denotes conjugation of partitions. 

Furthermore, the $\bd$-quotient of $\la$ is 
$\ladq=(\mu_0, \, \mu_1, \, \ldots , \, \mu_{\frac{d-1}{2}})$.
\end{theorem}
In our example, we find $\Dla^{(3)}=((2), \, (4), \,
(1,1,1,1))=(D((1)), \, (4), \, (4)^*)$, 
and $\ladq=\la^{(\bar{3})}=((1), \, (4))$.

\smallskip
Removing all the hooks of length divisible by $d$ in $\Dla$ (or,
equivalently, moving all the beads in the abacus of $\Dla$ as far up 
as possible on their respective runners), we obtain the $d$-core 
$\Dla_{(d)}$ of $\Dla$. Then we see (cf \cite[p. 26]{Morris-Yaseen}) 
that $\Dla_{(d)}=D(\lad)$, where $\lad$ is the $\bd$-core of $\la$ 
(which may also be obtained from $\la$ by removing all the bars of 
length divisible by $d$; the removal of such a bar corresponds to 
removing a pair of $d$-hooks from $\Dla$, one whose node is in
$S(\la)$, and its counterpart in the lower half of the diagram). 
In our example, we find $\Dla_{(3)}=(3, \, 1)=D((2))=D(\lad)$.

\smallskip
We define the {\emph{$\bar{d}$-quotient partition}} of $\la$ to be the
(uniquely defined) bar partition $\bdqla$ which has empty 
$\bd$-core, and $\bd$-quotient $\bdqla^{(\bd)}=\la^{(\bd)}$. 
The doubled partition $D(\bdqla)$ therefore has empty $d$-core, 
and $d$-quotient $D(\bdqla)^{(d)}=\Dla^{(d)}$. This proves that 
$D(\bdqla)$ is the {\emph{$d$-quotient partition}} of $\Dla$, 
which we write as $D(\bdqla)=q_d(\Dla).$ 

In the $d$-abacus of $D(\bdqla)$, the partition associated to each
runner is the same as for $\Dla$, but the corresponding $\beta$-sets
all have the same number of elements (which is the number of beads on the runners, and this is the same for each runner since $D(\bdqla)$ has empty $d$-core). In our example, we can take this
number to be 4, and we obtain 

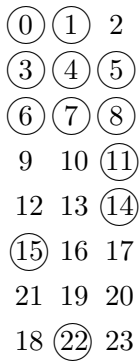
\setlength{\unitlength}{1mm}
\linethickness{0.2mm}
\begin{center}
\begin{picture}(30,48)

\put(0.5,1){$18$}
\put(0.5,7){$21$}
\put(0.5,13){$15$}
\put(0.5,19){$12$}
\put(1,25){$9$}
\put(1,31){$6$}
\put(1,37){$3$}
\put(1,43){$0$}
\put(6.5,1){$22$}
\put(6.5,7){$19$}
\put(6.5,13){$16$}
\put(6.5,19){$13$}
\put(6.5,25){$10$}
\put(7,31){$7$}
\put(7,37){$4$}
\put(7,43){$1$}
\put(12.5,1){$23$}
\put(12.5,7){$20$}
\put(12.5,13){$17$}
\put(12.5,19){$14$}
\put(12.5,25){$11$}
\put(13,31){$8$}
\put(13,37){$5$}
\put(13,43){$2$}

\put(2.4,14){\circle{5}}
\put(2,32.2){\circle{5}}
\put(2,38.2){\circle{5}}
\put(2,44.2){\circle{5}}
\put(8.2,2){\circle{5}}
\put(8,32.2){\circle{5}}
\put(8,38.2){\circle{5}}
\put(8,44.2){\circle{5}}
\put(14,32.2){\circle{5}}
\put(14,38.2){\circle{5}}
\put(14.4,20){\circle{5}}
\put(14.4,26){\circle{5}}

\end{picture}
\end{center}

\noindent
We therefore have $D(\bar{q}_3(\la))=q_3(\Dla)=(11, \, 5^2, \, 3, \, 1^6)$, 
and $\bar{q}_3(\la)=(10, \, 3, \, 2)$.

\smallskip
It is easy to see, using \cite[Theorem (4.3)]{Olsson-Combinatorics}, 
that since $\la$ and $\bdqla$ have the same $\bd$-quotient, 
there is a length-preserving bijection between the sets of bars of
length divisible by $d$ in $\la$ and $\bdqla$ respectively. 
In our example, both multisets of lengths are $\{ 12, \, 9, \, 6, \, 3, \, 3 \}$.

\smallskip





\section{Multisets of bar lengths}
We keep the notation as in Section \ref{doubling}, and we take any bar
partition $\mu$ (which we want to specialize to $\mu \in \{ \la, \,
\lad, \, \bdqla \}$). We define several subsets of the set 
$H(D(\mu))$ of hooks and the multiset ${\cal H}(\Dmu)$ 
of hook lengths in $\Dmu$.
This is illustrated in the following diagram:

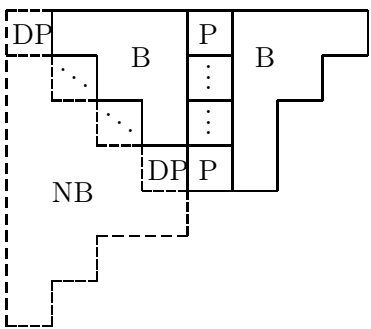
\setlength{\unitlength}{1.5mm}

\linethickness{0.2mm}

\begin{center}

\begin{picture}(30,30)

\put(4,28){\line(1,0){28}}
\put(4,28){\line(0,-1){4}}

\put(32,28){\line(0,-1){4}}
\put(28,24){\line(0,-1){4}}
\put(24,20){\line(0,-1){8}}

\put(16,28){\line(0,-1){16}}
\put(20,28){\line(0,-1){16}}

\put(16,16){\line(1,0){4}}
\put(16,20){\line(1,0){4}}
\put(16,24){\line(1,0){4}}

\put(4,24){\line(1,0){4}}
\put(28,24){\line(1,0){4}}
\put(8,20){\line(1,0){4}}
\put(8,20){\line(0,1){4}}
\put(24,20){\line(1,0){4}}
\put(12,16){\line(1,0){4}}
\put(12,16){\line(0,1){4}}
\put(16,12){\line(1,0){8}}

\dashline[+60]{1}(0,0)(4,0)
\dashline[+60]{1}(4,0)(4,4)
\dashline[+60]{1}(4,4)(8,4)
\dashline[+60]{1}(8,4)(8,8)
\dashline[+60]{1}(8,8)(16,8)
\dashline[+60]{1}(12,12)(16,12)
\dashline[+60]{1}(8,16)(12,16)
\dashline[+60]{1}(4,20)(8,20)
\dashline[+60]{1}(0,24)(4,24)
\dashline[+60]{1}(0,28)(4,28)
\dashline[+60]{1}(0,28)(0,0)
\dashline[+60]{1}(4,24)(4,20)
\dashline[+60]{1}(8,20)(8,16)
\dashline[+60]{1}(12,16)(12,12)
\dashline[+60]{1}(16,12)(16,8)

\put(0.5,25){DP}

\put(4.5,21){$\ddots$}
\put(8.5,17){$\ddots$}
\put(12.5,13){DP}
\put(17,13){P}
\put(17.5,17){\vdots}
\put(17.5,21){\vdots}
\put(17,25){P}
\put(22,23){B}
\put(4,11){NB}
\put(11,23){B}

\end{picture}
\end{center}
\noindent
We write $P(\Dmu)$ for the set of hooks corresponding to the parts of
$\mu$ (denoted by P above), and ${\cal P}(\Dmu)$ for the set of their lengths.

\noindent
We write $DP(\Dmu)$ for the set of hooks corresponding to the 
doubled parts of $\mu$ (denoted by DP above), and ${\cal{DP}}(\Dmu)$ 
for the set of their lengths.

\noindent
We write $B(\Dmu)$ for the set of hooks corresponding to bars in $\mu$
which are not parts (denoted by B above), and ${\cal B}(\Dmu)$ for the
multiset of their lengths.

\noindent
We write $NB(\Dmu)$ for the set of hooks corresponding to ``non-bars''
(denoted by NB above), i.e. the counterparts in the lower half of the 
Young diagram of the bars in $\mu$ which are not parts, 
and ${\cal{NB}}(\Dmu)$ for the multiset of their lengths. 
In particular, by construction, we have ${\cal{NB}}(\Dmu)={\cal B}(\Dmu)$.

\smallskip
We thus have the set equality $H(\Dmu)=P(\Dmu) \cup DP(\Dmu) \cup
B(\Dmu) \cup NB(\Dmu)$ and the multiset equality 
${\cal H}(\Dmu)={\cal P}(\Dmu) \cup {\cal{DP}}(\Dmu) \cup {\cal
  B}(\Dmu) \cup {\cal{NB}}(\Dmu)$. 
Note that $B(\Dmu)$ is the set of unmixed bars of type 1 and 
mixed bars (of type 3) in $\mu$, while $P(\Dmu)$ is the set of 
unmixed bars of type 2 in $\mu$ 
(see e.g. \cite[Section 4]{Olsson-Combinatorics}). 
In particular, we have for the multiset of bar lengths in 
$\mu$, that ${\cal B}(\mu)={\cal P}(\Dmu)  \cup {\cal B}(\Dmu)$.

\smallskip
For any $0 \leq i , \, j \leq d-1$, we have defined in 
Section \ref{doubling} subsets $H_{i \rightarrow j}(\Dmu)$, 
$H_{\{i\}}(\Dmu)$ and $H_{\{ij\}}(\Dmu)$ (if $i \neq j$) of
$H(\Dmu)$. Similarly, we define subsets $P_{i \rightarrow j}(\Dmu)$, 
$P_{\{ij\}}(\Dmu)$, $NP_{i \rightarrow j}(\Dmu)$, $NP_{\{ij\}}(\Dmu)$,
$B_{i \rightarrow j}(\Dmu)$, $B_{\{ij\}}(\Dmu)$, 
$NB_{i \rightarrow j}(\Dmu)$ and $NB_{\{ij\}}(\Dmu)$.

\smallskip
As before, for any hook $z$ in a partition we let $h(z)$ denote its length.
For any $z \in B(\Dmu)$, we denote by $z^*$ the counterpart of 
$z$ in $NB(\Dmu)$ (so that $h(z)=h(z^*)$). In particular, we have 
$NB(\Dmu)=B(\Dmu)^*$. For any $1 \leq i \leq d-1$, we let $i^*=d-i$ 
(so that, in the $d$-quotient 
$(D(\mu_0), \, \mu_1, \, \ldots , \, \mu_{d-1})$ of $\Dmu$, 
we have, for each $1 \leq i \leq d-1$, $\mu_{i^*}=\mu_i^*$). We also let $0^*=0$.

\begin{lemma}\label{starhook}
For any $z \in B(\Dmu)$, if $z \in B_{i \rightarrow j}(\Dmu)$ 
(for some $0 \leq i, \, j \leq d-1$), 
then $z^* \in NB_{j^* \rightarrow i^*}(\Dmu)$.
\end{lemma}

\begin{proof}
Suppose $z^* \in NB_{k \rightarrow \ell}(\Dmu)$. By definition, 
$z$ has hand residue $i$ and foot residue $[j+1]_d$, 
while $z^*$ has hand residue $k$ and foot residue $[\ell +1]_d$. 
In particular, considering the lengths $h(z)$ and $h(z^*)$ of 
$z$ and $z^*$, we have $h(z) \equiv i-j \modd$ and 
$h(z^*) \equiv k - \ell \modd$. But, since $h(z)=h(z^*)$, 
we have $i-j \equiv k - \ell \modd$.

Now, since $z \in B(\Dmu)$, the arm of $z$ is in a row which 
corresponds to a part of $\mu$, say the $r$th row of the Young 
diagram of $\Dmu$. Then, by construction of $\Dmu$, the counterpart 
$z^*$ of $z$ has its leg in the $r$th column of the Young diagram 
of $\Dmu.$ Also, by construction, this column is one node shorter 
than the $r$th row. It is then easy to see that, 
if the $r$th row has end residue $m$, then the $r$th column has end residue 
$[d-(m-1)]_d=[(d-m)+1]_d=[m^*+1]_d$ (indeed, we know that 
the residues increase from left to right in a row while 
they decrease from top to bottom in a column, and that the 
$r$th row and $r$th column intersect on a diagonal node 
of residue $0$). Therefore the foot residue of 
$z^*$ is $[i^*+1]_d$, where $i$ is the hand residue of $z$. 
Hence $\ell=i^*$. But then, from $h(z)=h(z^*)$, we obtain 
$i-j \equiv k-i^* \modd$, whence, since $i^*=d-i$, 
$k \equiv -j \modd$, and thus $k = j^*$.

\end{proof}

\begin{corollary}\label{star}
For any $0 \leq i, \, j \leq d-1$, we have 
$NB_{j^* \rightarrow i^*}(\Dmu)= B_{i \rightarrow j}(\Dmu)^*$ 
and ${\cal{NB}}_{j^* \rightarrow i^*}(\Dmu)= {\cal B}_{i \rightarrow j}(\Dmu)$.
\end{corollary}

We now prove a symmetry property on the number of beads in the abacus
of $\Dmu$. We suppose that the $d$-abacus of $\Dmu$ is 
{\emph{minimally normalized}}, i.e. that the $\beta$-set for $\Dmu$ 
used to build the abacus has a multiple of $d$ elements, and is
minimal with respect to this property. 
For each $0 \leq i \leq d-1$, we write $x_i$ for the number of 
beads on the $i$th runner of the (minimally normalized) $d$-abacus of $\Dmu$.

\begin{lemma}\label{xi}
For any $0 \leq i, \, j \leq d-1$, we have $x_i+x_{i^*}=x_j+x_{j^*}$.

\end{lemma}

In the example of Section \ref{doubling}, we have $x_0=3$, $x_1=2$ and
$x_2=4$, whence $x_0+x_{0^*}=x_1+x_{1^*}=x_2+x_{2^*}=6$.

\begin{proof}
Let $a$ be the largest part of $\mu$, and $[a]_d$ its $d$-residue. 
Consider the {\emph{rim}} $R$ of the Young diagram of $\Dmu$. 
Then $R$ is composed of $a+1$ horizontal segments (of length 1) and 
$a$ vertical segments. 

We extend the rim horizontally to the top right and vertically 
to the bottom left as follows. Suppose $(k-1)d < a \leq kd$ for some 
$k \in \N$. If $a+1 \leq kd$ (i.e. $[a]_d \neq 0$), then extend $R$ to
$\widetilde{R}$ which has $kd$ horizontal segments and $kd$ vertical ones. 
If $a+1 > kd$ (i.e. $a=kd$), then extend $R$ to $\widetilde{R}$ which has
$(k+1)d$ horizontal segments and $(k+1)d$ vertical ones. 
In particular, $\widetilde{R}$ always has $\ell d$ vertical segments and 
$\ell d$ horizontal ones, with $\ell d >a$. This implies that, while 
the horizontal extension of $R$ may be empty (if $a+1=kd$), the
vertical one never is. In fact, the horizontal extension is always 
exactly one segment shorter than the vertical one, which is 
$\ell d -a$ segment long ($\ell d-a \neq 0$, 
and $\ell d - a - 1=0 \Longleftrightarrow a+1 = kd$).

Label the vertical segment at the end of the first row by its 
$d$-residue $[a]_d$. Complete the labelling of each segment of $\widetilde{R}$ 
by residues modulo $d$ by increasing by 1 for each step to the top or 
right, and decreasing by 1 for each step to the bottom or left. In 
particular, a row has end residue $j$ if and only if the corresponding
vertical segment of $\widetilde{R}$ is labelled by $j$, while a column has end 
residue $j$ if and only if the corresponding horizontal segment of 
$\widetilde{R}$ is labelled by $[j-1]_d$.

For each $0 \leq j \leq d-1$, let $V_j$ (respectively $H_j$) be the 
set of vertical (respectively horizontal) segments of $\widetilde{R}$ 
labelled by $j$. We thus have
\begin{equation}\label{V}|V_j|=
\left\{ \begin{array}{cl} x_j & \mbox{if} \; 
(k-1)d<a<kd \\ x_j+1 & \mbox{if} \; a=kd \end{array} \right. 
 \; \; \; (0 \leq j \leq d-1). \end{equation}
By construction of $\Dmu$ (see also the proof of Lemma
\ref{starhook}), we see that the horizontal segment at the bottom 
of the first column is labelled by $[a]_d^*$. Since the vertical 
extension of $R$ has $\ell d -a$ segments, the bottom one is labelled 
by $[[a]_d^*-(\ell d -a)]_d=0$. And, since the horizontal extension 
has $\ell d -a -1$ segments, the last one (if it exists, i.e. if 
$\ell d -a -1 \neq 0$) is labelled by $[a + \ell d - a -1]_d=d-1$.

Following $\widetilde{R}$ from bottom left to top right, the labelling residues 
increase by one at each step, going (by the above) from $0$ to $d-1$, 
and this $2 \ell$ times (since $\widetilde{R}$ has $2 \ell d$ segments). 
This shows that, for each $0 \leq j \leq d-1$, the total number of 
segments of $\widetilde{R}$ labelled by $j$ is $2 \ell$, i.e.
\begin{equation}\label{VandH} |V_j| + |H_j| = 2 \ell \; \; \; \; \; 
(0 \leq j \leq d-1). \end{equation}

The nodes on the diagonal all have residue 0. Thus the column of 
nodes which corresponds to the parts of $\mu$ has end residue 1 
(since immediately to the right of the diagonal). Hence the
corresponding horizontal segment is labelled by 0. 
We have the following picture

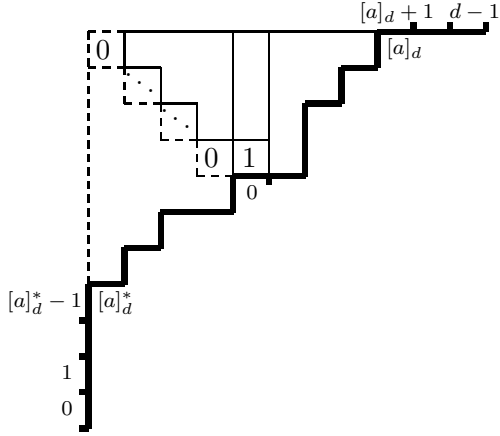
\setlength{\unitlength}{1.2mm}

\linethickness{0.2mm}

\begin{center}

\begin{picture}(54,50)

\put(14,48){\line(1,0){28}}
\put(14,48){\line(0,-1){4}}

\put(26,48){\line(0,-1){16}}
\put(30,48){\line(0,-1){16}}

\put(26,36){\line(1,0){4}}

\put(14,44){\line(1,0){4}}
\put(18,40){\line(1,0){4}}
\put(18,40){\line(0,1){4}}
\put(22,36){\line(1,0){4}}
\put(22,36){\line(0,1){4}}

\dashline[+60]{1}(22,32)(26,32)
\dashline[+60]{1}(18,36)(22,36)
\dashline[+60]{1}(14,40)(18,40)
\dashline[+60]{1}(10,44)(14,44)
\dashline[+60]{1}(10,48)(14,48)
\dashline[+60]{1}(10,48)(10,20)
\dashline[+60]{1}(14,44)(14,40)
\dashline[+60]{1}(18,40)(18,36)
\dashline[+60]{1}(22,36)(22,32)

\put(10.8,45){$0$}

\put(13.8,41){$\ddots$}

\put(17.8,37){$\ddots$}

\put(22.8,33){$0$}

\put(27,33){$1$}

\linethickness{0.7mm}

\put(42,48){\line(1,0){12}}

\put(42,48){\line(0,-1){4}}
\put(38,44){\line(0,-1){4}}
\put(34,40){\line(0,-1){8}}

\put(26,32){\line(1,0){8}}
\put(38,44){\line(1,0){4}}

\put(34,40){\line(1,0){4}}

\put(14,20){\line(0,1){4}}
\put(18,24){\line(0,1){4}}
\put(26,28){\line(0,1){4}}
\put(10,20){\line(1,0){4}}
\put(14,24){\line(1,0){4}}
\put(18,28){\line(1,0){8}}

\put(10,20){\line(0,-1){16}}

\put(10,16){\line(-1,0){1}}
\put(10,12){\line(-1,0){1}}
\put(10,8){\line(-1,0){1}}
\put(10,4){\line(-1,0){1}}

\put(46,48){\line(0,1){1}}
\put(50,48){\line(0,1){1}}
\put(54,48){\line(0,1){1}}

\put(30,32){\line(0,-1){1}}

\put(27.5,29.5){\scriptsize{{\bf{$0$}}}}

\put(43,45.5){\scriptsize{{\bf{$[a]_d$}}}}
\put(40,49.5){\scriptsize{{\bf{$[a]_d+1$}}}}
\put(50,49.5){\scriptsize{{\bf{$d-1$}}}}
\put(11,17.5){\scriptsize{{\bf{$[a]_d^*$}}}}
\put(1,17.5){\scriptsize{{\bf{$[a]_d^*-1$}}}}
\put(7,5.5){\scriptsize{{\bf{$0$}}}}
\put(7,9.5){\scriptsize{{\bf{$1$}}}}

\end{picture}
\end{center}

Now the portions of $\widetilde{R}$ to the right and to the left of this 0 
(except the bottom left segment) are symmetric (by construction 
of $\Dmu$, and by the above considerations on the horizontal and 
vertical extensions). The vertical segments of one are in bijection 
with the horizontal ones of the other, and any label $j$ is sent to 
a label $j^*$. Adding the (horizontal) 0 in the middle and the 
(vertical) 0 at the bottom, this proves that $|V_j|=|H_{j^*}|$ for 
each $0 \leq j \leq d-1$. Together with (\ref{VandH}), this yields
$$|V_j|+|V_{j^*}|=|V_j|+|H_j|=2 \ell \; \; \; \; \; (0 \leq j \leq d-1).$$
Using (\ref{V}), this implies the result.
\end{proof}

This has several important consequences in our context. Let $\la$ be any bar
partition, and $\bdqla$ be its $\bd$-quotient partition. Following
\cite[Theorem 4.7]{BeGrOl}, we define, for each hook $z \in
H(D(\bdqla))$, the modified hook length $\overline{h}(z)$ by
$\overline{h}(z)=h(z)+(x_i-x_j)d$ if $z$ has hand residue $i$ and foot
residue $[j+1]_d$. We then write $\overline{\cal H}(D(\bdqla))= \{ |
\overline{h}(z)| \; | \, z \in H(D(\bdqla)) \}$ (note that the same
multiset is denoted by abs$(\overline{\cal H}(D(\bdqla)))$ in
\cite{BeGrOl}). Subsets of modified hook lengths are defined 
similarly for the subsets of hooks we introduced earlier.

\begin{corollary}\label{modifiedlength}
For any bar partition $\la$ and any $z \in B(D(\bdqla)$, 
we have $\overline{h}(z)=\overline{h}(z^*)$. In particular, 
for any $0 \leq i, \, j \leq d-1$, we have 
$$\overline{{\cal B}}_{i \rightarrow j} (D(\bdqla))
=\overline{{\cal{NB}}}_{j^* \rightarrow i^*} (D(\bdqla)).$$
\end{corollary}

\begin{proof}
Suppose $z \in  B_{i \rightarrow j} (D(\bdqla))$ 
(for some $0 \leq i, \, j \leq d-1$). Then, by 
Lemma \ref{starhook}, $z^* \in  NB_{j^* \rightarrow i^*}
(D(\bdqla))$. 
We thus have $\overline{h}(z)=h(z)+(x_i-x_j)d$ and 
$\overline{h}(z^*)=h(z^*)+(x_{j^*}-x_{i^*})d$. 
But $h(z)=h(z^*)$ and, by Lemma \ref{xi}, $x_i-x_j=x_{j^*}-x_{i^*}$. 
Hence $\overline{h}(z)=\overline{h}(z^*)$. Corollary \ref{star} 
concludes the proof.
\end{proof}

\begin{corollary}\label{doubleparts}
For any bar partition $\la$ and any $z \in P(D(\bdqla))$, we let $z
^{\times 2}$ be the corresponding element in $DP(D(\bdqla))$, satisfying
$h(z^{\times 2})=2h(z).$ We then have 
$\overline{h}(z^{\times 2})=2\overline{h}(z)$.
\end{corollary}

\begin{proof}
Take any  $z \in P(D(\bdqla)).$ The proof of
Lemma \ref{xi} shows the following: if  $h(z) \equiv i \modd$ for some 
$0 \leq i \leq d-1$, then $z \in P_{i \rightarrow 0}(D(\bdqla))$
and $z^{\times 2} \in DP_{i \rightarrow i^*}(D(\bdqla))$. 
By definition, we therefore get that 
$\overline{h}(z)=  h(z) + d (x_i-x_{0})$ and 
$\overline{h}(z^{\times 2})= 2h(z) + d (x_i-x_{i^*})$. Now, by Lemma \ref{xi}, we see
that $x_i-x_{i^*}=2(x_i-x_0),$  so that 
$\overline{h}(z^{\times 2})= 2h(z) + d (x_i-x_{i^*})=  2h(z) + 2d(x_i-x_0)= 2 \overline{h}(z)$.
\end{proof}

\begin{corollary}\label{distinctparts}
For any bar partition $\la$, the elements of 
$ \overline{{\cal P}}(D(\bdqla)$ are distinct, 
and the elements of $ \overline{{\cal DP}}(D(\bdqla)$ are distinct.

\end{corollary}

\begin{proof}
Suppose $z \in P_{i \rightarrow 0}(D(\bdqla))$, 
$z' \in P_{j \rightarrow 0}(D(\bdqla))$ and 
$|\overline{h}(z)|=|\overline{h}(z')|$. We want to show that
$h(z)=h(z'),$ which then implies $z=z'.$
We see from the definition that $|\overline{h}(z)| \equiv \pm i \modd$ 
and $|\overline{h}(z')| \equiv \pm j \modd$. 
Thus we must have $j=i$ or $j=i^*$. 
If $i=0$, then $j=0$, so that 
$\overline{h}(z)=h(z)$ and $\overline{h}(z')=h(z')$ and we are done. 
If we have $j=i \neq 0$, then if  
$\overline{h}(z)=\overline{h}(z')$ we obviously have that $h(z)=h(z').$ 
If $\overline{h}(z)=-\overline{h}(z')$ we get $i \equiv -i \modd,$ which is
impossible. 
Consider the case that $j=i^*$, $h(z) \equiv i \modd$ and 
$h(z') \equiv i^*\modd$. Then $\overline{h}(z)=-\overline{h}(z'),$ i.e. 
$h(z)+d(x_i-x_0)=-h(z')-d(x_{i^*}-x_0).$ Thus $h(z)+h(z')=d(x_0-x_i)+d(x_0-x_{i^*}).$ 
But the right hand side of this is 0, by Lemma  \ref{xi},  
which is impossible.

We have now shown that the elements of $ \overline{{\cal P}}(D(\bdqla)$ are 
distinct, and then Corollary \ref{doubleparts} shows that the
elements of $ \overline{{\cal DP}}(D(\bdqla)$ are distinct.

\end{proof}

\section{Main result}
We are now in position to prove our main result. 
Recall that, if $\la$ is a bar partition, then the multiset of bar lengths 
in $\la$ is ${\cal B}(\la)= {\cal B}(\Dla) \cup {\cal P}(\Dla)$. We will also write $\cP(\la)$ for the set $\cP(\Dla)$ of parts of $\la$. Also $ \overline{\cal B}(\bdqla)$ is the multiset of absolute 
values of the modified hook lengths $\overline{h}(z),$ $z \in  B(\bdqla). $ Finally, for any multiset $A$, we denote by $A^{\times 2}$ the multiset given by $A^{\times 2}=\{2a \; | \; a \in A \}$.

\begin{theorem}\label{main}
For any bar partition $\la$ and any odd integer $d \geq 3$, we 
have $${\cal B}(\la)={\cal B}(\lad) \cup \widetilde{\cal B}(\bdqla),$$
where $ \widetilde{\cal B}(\bdqla)= \ov{\cB}(\bdqla) \setminus [\cP(\lad) \cap \ov{\cP}(\bdqla)] \cup [\cP(\lad) \cap \ov{\cP}(\bdqla)]^{\times 2}$.
\end{theorem} 

\begin{proof}

We write $$B(\la)= \bigcup_{0 \leq i \leq d-1} B_{\{i\}}(\la) \cup 
\bigcup_{0 \leq i < j \leq d-1} B_{\{ij\}}(\la)$$
(where we have $B_{\{i\}}(\la)= B_{i \rightarrow i}(\la)=
B_{i \rightarrow i}(\Dla) \cup P_{i \rightarrow i}(\Dla)$ 
and $B_{\{ij\}}(\la)=B_{i \rightarrow j}(\la) \cup B_{j \rightarrow
  i}(\la)=B_{\{ij\}}(\Dla) \cup P_{\{ij\}}(\Dla)$).
Now $\bigcup_{0 \leq i \leq d-1} B_{\{i\}}(\la)$ is exactly the set 
of bars of length divisible by $d$ in $\la$. By construction 
(see Section \ref{doubling}), we thus have
$$\bigcup_{0 \leq i \leq d-1} \cB_{\{i\}}(\la) 
= \bigcup_{0 \leq i \leq d-1} \cB_{\{i\}}(\bdqla) 
= \bigcup_{0 \leq i \leq d-1} \overline{\cB}_{\{i\}}(\bdqla),$$
and$$\bigcup_{0 \leq i \leq d-1} \cB_{\{i\}}(\lad)= \emptyset.$$

We also want to examine separately the case of parts of length divisible by $d$ in $\la$. These are the bars of type 2 (see \cite[Section 4]{Olsson-Combinatorics}) of length divisible by $d$ in $\la$, and, if we write $\ladq=\bdqla^{(\bd)}=(\mu_0, \, \mu_1, \, \ldots , \, \mu_{\frac{d-1}{2}})$, then they correspond bijectively to the parts of $\mu_0$ (and are $d$ times as long). Since $\bdqla$ has the same $\bd$-quotient as $\la$, we see that it has the same parts of length divisible by $d$.

\smallskip

To prove our result, it is now sufficient to consider the bars 
whose length is not divisible by $d$, i.e. which correspond to a bead 
and an empty spot on distinct runners $i$ and $j$ in the abacus. 
We distinguish between three cases, corresponding to the three
possible cardinalities of the set $\{i, \, j, \, i^*, \, j^*\}$.

\smallskip

\noindent
{\bf{Case (1).}} Take any $0 \leq i < j \leq d-1$ such that 
$|\{i, \, j, \, i^*, \, j^*\}|=4$. In particular, $j \neq i^*$, so 
that $DP_{\{ij\}}(\Dla)=\emptyset=DP_{\{i^*j^*\}}(\Dla)$, and 
$0 \not \in \{i, \, j, \, i^*, \, j^*\}$, so that 
$P_{\{ij\}}(\Dla)=\emptyset=P_{\{i^*j^*\}}(\Dla)$ 
(and similarly for $\Dlad$ and $\Dbdqla$).

By \cite[Theorem 4.7]{BeGrOl} (which can be refined to pairs of 
runners using the proof of \cite[Theorem 3.2]{BeGrOl}), we have
$$\cH_{\{ij\}}(\Dla)=\cH_{\sij}(\Dlad) \cup \ov{\cH}_{\sij}(\Dbdqla).$$
Now $\cH_{\{ij\}}(\Dla)=\cHij(\Dla) \cup \cHji(\Dla)$, and we have
$$\cHij(\Dla)=\cBij(\Dla) \cup \cNBij(\Dla) $$and$$ \cHji(\Dla)
=\cBji(\Dla) \cup \cNBji(\Dla).$$
Also, by Corollary \ref{star},$$\cNBij(\Dla) 
= \cBjistar(\Dla) \; \; \mbox{and} \; \; \cNBji(\Dla) 
= \cBijstar(\Dla),$$whence$$\cH_{\{ij\}}(\Dla)
=\cB_{\sij}(\Dla) \cup \cB_{\sijs}(\Dla)
=\cB_{\sij}(\la) \cup \cB_{\sijs}(\la).$$
Similarly, $\cH_{\{ij\}}(\Dlad)=\cB_{\sij}(\lad) \cup
\cB_{\sijs}(\lad)$, and, using Corollary \ref{modifiedlength}, 
$\ov{\cH}_{\{ij\}}(\Dbdqla)=\ov{\cB}_{\sij}(\bdqla) \cup
\ov{\cB}_{\sijs}(\bdqla)$. Hence, in this case,
$$\cB_{\sij}(\la) \cup \cB_{\sijs}(\la)
=\cB_{\sij}(\lad) \cup \cB_{\sijs}(\lad) 
\cup \ov{\cB}_{\sij}(\bdqla) \cup \ov{\cB}_{\sijs}(\bdqla).$$

\smallskip

\noindent
{\bf{Case (2)}.} Take any $0 \leq i < j \leq d-1$ such that 
$|\{i, \, j, \, i^*, \, j^*\}|=3$. This means that $i=i^*=0$, and 
$j \neq j^*$ (in particular, $j \neq 0$ and $j \neq i^*$), so that 
$P_{i \rightarrow j}(\Dla)= \emptyset$ and 
$DP_{\{ij\}}(\Dla)=\emptyset=DP_{\{i^*j^*\}}(\Dla)$ 
(and similarly for $\Dlad$ and $\Dbdqla$).

This time, we have$$\cHij(\Dla)=\cH_{0 \rightarrow j}(\Dla)
=\cB_{0 \rightarrow j}(\Dla) \cup \cNB_{0 \rightarrow j}(\Dla)$$
and
$$\cH_{j \rightarrow 0}(\Dla)=\cP_{j \rightarrow 0}(\Dla) 
\cup \cB_{j \rightarrow 0}(\Dla) \cup \cNB_{j \rightarrow 0}(\Dla).$$
Also, by Corollary \ref{star}, 
we have $$\cNB_{0 \rightarrow j}(\Dla)=\cB_{j^* \rightarrow 0}(\Dla)
\; \;  \mbox{and} \; \; \cNB_{j \rightarrow 0}(\Dla)
=\cB_{0 \rightarrow j^*}(\Dla),$$ whence $\cH_{\{0j\}}(\Dla)= 
\cP_{\{0j\}}(\Dla) \cup \cB_{\{0j\}}(\Dla) \cup \cB_{\{0j^*\}}(\Dla)$.

Now $0 < j^* \leq d-1$, and $(0, \, j^*)$ satisfies the same condition
as $(0, \, j)$ (and, in fact, $\{0, \, j, \, 0^*, \, j^*\}=
\{0, \, j^*, \, 0^*, \, (j^*)^*\}$). Thus we also have 
$\cH_{\{0j^*\}}(\Dla)= \cP_{\{0j^*\}}(\Dla) \cup \cB_{\{0j^*\}}(\Dla) 
\cup \cB_{\{0j\}}(\Dla)$.

Writing $\cH^{\la}_{\{0jj^*\}}$ for $\cH_{\{0j\}}(\Dla) \cup \cH_{\{0j^*\}}(\Dla)$, and using similar notation for parts and bars, we hence obtain 
$$\cH^{\la}_{\{0jj^*\}} = \cP^{\la}_{\{0jj^*\}} \cup 2 \cB^{\la}_{\{0jj^*\}},$$
where, for any multiset $A$, we write $2A$ for $A \cup A$. Similarly, and with analogous notation, we have$$\cH^{\lad}_{\{0jj^*\}} = \cP^{\lad}_{\{0jj^*\}} \cup 2 \cB^{\lad}_{\{0jj^*\}}.$$And, using Corollary \ref{modifiedlength}, 
we also obtain$$\ov{\cH}_{\{0jj^*\}}^{\bdqla} 
= \ov{\cP}_{\{0jj^*\}}^{\bdqla}  \cup   
2 \ov{\cB}_{\{0jj^*\}}^{\bdqla}.$$
Finally, by \cite[Theorem 4.7]{BeGrOl}, we have
$$\cH^{\la}_{\{0jj^*\}}= \cH^{\lad}_{\{0jj^*\}} \cup \ov{\cH}_{\{0jj^*\}}^{\bdqla}  .$$
Rewriting this equality using the expressions we found above, we obtain
$$\begin{array}{c}
\cP^{\la}_{\{0jj^*\}} \cup 2 \cB^{\la}_{\{0jj^*\}}  = \cP^{\lad}_{\{0jj^*\}} \cup 2 \cB^{\lad}_{\{0jj^*\}} \cup  \ov{\cP}_{\{0jj^*\}}^{\bdqla}  \cup   
2 \ov{\cB}_{\{0jj^*\}}^{\bdqla} \\  = \cP^{\lad}_{\{0jj^*\}} \circ  \ov{\cP}_{\{0jj^*\}}^{\bdqla}  \cup 2 [ (\cP^{\lad}_{\{0jj^*\}} \cap  \ov{\cP}_{\{0jj^*\}}^{\bdqla} ) \cup \cB^{\lad}_{\{0jj^*\}}  \cup \ov{\cB}_{\{0jj^*\}}^{\bdqla} ],
\end{array}$$
where $\circ$ denotes symmetric difference.

Now any multiset $Q$ has a unique decomposition of the form $Q=R + 2S$, where $R$ and $S$ are sub-multisets, and the elements of $R$ are distinct. The elements of $ \cP^{\lad}_{\{0jj^*\}}$ are distinct since $\lad$ is a bar partition, and those of $\ov{\cP}_{\{0jj^*\}}^{\bdqla} $ are distinct by Corollary \ref{distinctparts}, whence the elements of $\cP^{\lad}_{\{0jj^*\}} \circ  \ov{\cP}_{\{0jj^*\}}^{\bdqla} $ are distinct. This implies that $\cP^{\la}_{\{0jj^*\}}=\cP^{\lad}_{\{0jj^*\}} \circ  \ov{\cP}_{\{0jj^*\}}^{\bdqla} $ and $\cB^{\la}_{\{0jj^*\}} = (\cP^{\lad}_{\{0jj^*\}} \cap  \ov{\cP}_{\{0jj^*\}}^{\bdqla} ) \cup \cB^{\lad}_{\{0jj^*\}}  \cup \ov{\cB}_{\{0jj^*\}}^{\bdqla} $. In particular, we obtain, for each $1 \leq j \leq \frac{d-1}{2}$,
$$\cB_{\{0jj^*\}}(\la) =
\cB_{\{0jj^*\}}(\lad)  \cup
\ov{\cB}_{\{0jj^*\}}(\bdqla) \setminus (\cP^{\lad}_{\{0jj^*\}} \cap  \ov{\cP}_{\{0jj^*\}}^{\bdqla} ) .$$

\smallskip

\noindent
{\bf{Case (3).}} Finally, take any $0 \leq i < j \leq d-1$ such that 
$|\{i, \, j, \, i^*, \, j^*\}|=2$. This means that $i \neq 0$, and 
$j = i^*$. In particular, no part is going to appear in this way, 
while all the doubled parts non-divisible by $d$ will. We have
$$\cH_{i \rightarrow i^*}(\Dla)
= \cDP_{i \rightarrow i^*}(\Dla) \cup \cB_{i \rightarrow i^*}(\Dla) 
\cup \cNB_{i \rightarrow i^*}(\Dla)$$ and $$\cH_{i^* \rightarrow
  i}(\Dla)
= \cDP_{i^* \rightarrow i}(\Dla) \cup \cB_{i^* \rightarrow i}(\Dla) 
\cup \cNB_{i^* \rightarrow i}(\Dla),$$
whence, by Corollary \ref{star}, $\cH_{ \{ii^*\} }(\Dla)=
\cDP_{ \{ii^*\} }(\Dla) \cup 2 \,  \cB_{ \{ii^*\} }(\Dla)$. 
Similarly, $\cH_{ \{ii^*\} }(\Dlad)=\cDP_{ \{ii^*\} }(\Dlad) \cup 2 
\,  \cB_{ \{ii^*\} }(\Dlad)$ and, by Corollary \ref{modifiedlength}, 
$\ov{\cH}_{ \{ii^*\} }(\Dbdqla)=\ov{\cDP}_{ \{ii^*\} }(\Dbdqla) 
\cup 2 \,  \ov{\cB}_{ \{ii^*\} }(\Dbdqla)$.

Applying \cite[Theorem 4.7]{BeGrOl} and using similar notation to that used in Case (2), we obtain
$$\begin{array}{c}
\cDP^{\la}_{\{ii^*\}} \cup 2 \cB^{\la}_{\{ii^*\}}  = \cDP^{\lad}_{\{ii^*\}} \cup 2 \cB^{\lad}_{\{ii^*\}} \cup  \ov{\cDP}_{\{ii^*\}}^{\bdqla}  \cup   
2 \ov{\cB}_{\{ii^*\}}^{\bdqla} \\  = \cDP^{\lad}_{\{ii^*\}} \circ  \ov{\cDP}_{\{ii^*\}}^{\bdqla}  \cup 2 [ (\cDP^{\lad}_{\{ii^*\}} \cap  \ov{\cDP}_{\{ii^*\}}^{\bdqla} ) \cup \cB^{\lad}_{\{ii^*\}}  \cup \ov{\cB}_{\{ii^*\}}^{\bdqla} ].
\end{array}$$

Now the elements of $ \cDP^{\lad}_{\{ii^*\}}$ are distinct, and those of $\ov{\cDP}_{\{ii*\}}^{\bdqla} $ are distinct by Corollary \ref{distinctparts}, whence the elements of $\cDP^{\lad}_{\{ii^*\}} \circ  \ov{\cDP}_{\{ii^*\}}^{\bdqla} $ are distinct. This implies that $$\cB^{\la}_{\{ii^*\}} = (\cDP^{\lad}_{\{ii^*\}} \cap  \ov{\cDP}_{\{ii^*\}}^{\bdqla} ) \cup \cB^{\lad}_{\{ii^*\}}  \cup \ov{\cB}_{\{ii^*\}}^{\bdqla} .$$

Note that (using Corollary \ref{doubleparts}) we have 
$$\cDP^{\lad}_{\{ii^*\}} \cap  \ov{\cDP}_{\{ii^*\}}^{\bdqla}=[\cP^{\lad}_{\{0ii^*\}} \cap  \ov{\cP}_{\{0ii^*\}}^{\bdqla}]^{\times 2}:= \{ 2h \; | \; h \in \cP^{\lad}_{\{0ii^*\}} \cap  \ov{\cP}_{\{0ii^*\}}^{\bdqla} \}.$$
Thus, for each $1 \leq i \leq \frac{d-1}{2}$, we have
$$ \cB_{ \{ii^*\} }(\la)= \cB_{ \{ii^*\} }(\lad) 
\cup  \ov{\cB}_{ \{ii^*\} }(\bdqla) \cup [\cP^{\lad}_{\{0ii^*\}} \cap  \ov{\cP}_{\{0ii^*\}}^{\bdqla}]^{\times 2} .$$

\smallskip

Finally, we see that our three cases cover all the bars between 
any pair of (distinct) runners, because
$$ \bigcup_{0 \leq i < j \leq d-1, \atop |\{i,  j,  i^*,  j^*\}|=4} 
\{i,  j\} \cup \{i^*,  j^*\} \cup  
\bigcup_{1 \leq  j \leq \frac{d-1}{2}} \{0,  j\} \cup \{0,  j^*\} 
\cup  \bigcup_{1 \leq  i \leq \frac{d-1}{2}} \{i,  i^*\} 
=   \bigcup_{1 \leq i <  j \leq d-1} \{i,  j\}.$$

When we take the union of all the subsets of bars we computed, we see that all the parts of $\lad$ and $\bdqla$ which are not divisible by $d$ appear exactly once in case (2) and their doubles once in case (3). Since $\lad$ has no part divisible by $d$, we therefore obtain, together with the case of bars on a single runner,
$$ \cB(\la) =  [\cB(\lad) \cup \ov{\cB}(\bdqla)] \setminus [\cP(\lad) \cap \ov{\cP}(\bdqla)] \cup [\cP(\lad) \cap \ov{\cP}(\bdqla)]^{\times 2}$$
i.e. $ \cB(\la) =  \cB(\lad) \cup \widetilde{\cB}(\bdqla)$, as claimed.
 \end{proof}

For any bar partition $\mu$, we denote by $m(\mu)$ the number of parts of $\mu$.

\begin{corollary}\label{inclusions}
For any bar partition $\la$ and any odd integer $d \geq 3$, the following hold:
\begin{enumerate}

\item 
${\cal B}(\lad) \subset {\cal B}(\la)$,

\item
$\ov{{\cal B}}(\bdqla) \subset {\cal B}(\la)$,

\item
$\cP(\la)=\cP(\lad) \circ \ov{\cP}(\bdqla)$ ,

\item
$m(\lad)+m(\bdqla)=m(\la)+2|\cP(\lad) \cap \ov{\cP}(\bdqla)|,$

\item
$ \cB(\la) =  \cB(\lad) \cup \ov{\cB}(\bdqla)$ {\rm if and only if} $m(\la)=m(\lad) + m(\bdqla)$.

\end{enumerate}

\end{corollary}

\begin{proof}
(1) is immediate from Theorem \ref{main}. 
To obtain (2), one just has to rewrite the result as $\cB(\la)=
\ov{{\cal B}}(\bdqla) \cup   \left( \cB(\lad)  \setminus [\cP(\lad)
  \cap \ov{\cP}(\bdqla)] \right) \cup [\cP(\lad) \cap
\ov{\cP}(\bdqla)]^{\times 2}$. 
(3) is visible in the proof of Theorem \ref{main}: the parts of length
divisible by $d$ are the same in $\la$ and $\bdqla$ (while $\lad$ has
none), and those of length not divisible by $d$ in $\la$ are examined 
in Case (2) of the proof. Looking at the cardinalities of the sets 
involved, (4) is a direct consequence of (3). 
By (4), $m(\la)=m(\lad) + m(\bdqla)$ if and only if 
$\cP(\lad) \cap \ov{\cP}(\bdqla) = \emptyset$. This, in turn, is 
by Theorem \ref{main} equivalent to $ \cB(\la) =  \cB(\lad) \cup \ov{\cB}(\bdqla)$ (as $A^{\times 2} \neq A$ for any non-empty multiset $A$).

\end{proof}

\noindent
{\bf{Remark:}} Note that the situation given in (5) above does occur, 
for instance in the example we introduced in Section \ref{doubling}. 

\medskip



We now illustrate Theorem \ref{main} by an explicit example. As the above remark shows, the example we introduced in Section \ref{doubling} doesn't fully illustrate the extent of Theorem \ref{main}. We therefore consider instead the bar partition $\la=(13, \, 10, \, 4)$ of $n=27$, and $d=3$. Below is the shifted diagram of $\la$, filled in with the corresponding bar lengths:

\setlength{\unitlength}{1.3mm}

\linethickness{0.2mm}

\begin{center}

\begin{picture}(52,18)

\put(0,16){\line(1,0){52}}
\put(0,12){\line(1,0){52}}

\put(4,8){\line(1,0){40}}
\put(8,4){\line(1,0){16}}

\put(0,16){\line(0,-1){4}}
\put(4,16){\line(0,-1){8}}
\put(8,16){\line(0,-1){12}}
\put(12,16){\line(0,-1){12}}
\put(16,16){\line(0,-1){12}}
\put(20,16){\line(0,-1){12}}
\put(24,16){\line(0,-1){12}}
\put(28,16){\line(0,-1){8}}
\put(32,16){\line(0,-1){8}}
\put(36,16){\line(0,-1){8}}
\put(40,16){\line(0,-1){8}}
\put(44,16){\line(0,-1){8}}
\put(48,16){\line(0,-1){4}}
\put(52,16){\line(0,-1){4}}

\put(49.5,13){1}
\put(45.5,13){2}
\put(41.5,13){4}
\put(37.5,13){5}
\put(33.5,13){6}
\put(29.5,13){7}
\put(25.5,13){8}
\put(20.5,13){10}
\put(16.5,13){11}
\put(12.5,13){12}
\put(8.5,13){\bf{13}}
\put(4.5,13){17}
\put(0.5,13){23}

\put(41.5,9){1}
\put(37.5,9){2}
\put(33.5,9){3}
\put(29.5,9){4}
\put(25.5,9){5}
\put(21.5,9){\bf{7}}
\put(17.5,9){8}
\put(13.5,9){9}
\put(8.5,9){\bf{10}}
\put(4.5,9){\bf{14}}

\put(21.5,5){\bf{1}}
\put(17.5,5){\bf{2}}
\put(13.5,5){3}
\put(9.5,5){\bf{4}}

\end{picture}

\end{center}

We then have $\bar{c}_3(\la)=(7, \, 4, \, 1)$, with corresponding shifted diagram:

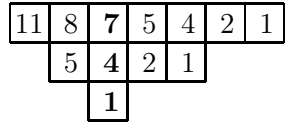
\setlength{\unitlength}{1.3mm}

\linethickness{0.2mm}

\begin{center}

\begin{picture}(28,18)

\put(0,16){\line(1,0){28}}
\put(0,12){\line(1,0){28}}

\put(4,8){\line(1,0){16}}
\put(8,4){\line(1,0){4}}

\put(0,16){\line(0,-1){4}}
\put(4,16){\line(0,-1){8}}
\put(8,16){\line(0,-1){12}}
\put(12,16){\line(0,-1){12}}
\put(16,16){\line(0,-1){8}}
\put(20,16){\line(0,-1){8}}
\put(24,16){\line(0,-1){4}}
\put(28,16){\line(0,-1){4}}

\put(25.5,13){1}
\put(21.5,13){2}
\put(17.5,13){4}
\put(13.5,13){5}
\put(9.5,13){\bf{7}}
\put(5.5,13){8}
\put(0.5,13){11}

\put(17.5,9){1}
\put(13.5,9){2}
\put(9.5,9){\bf{4}}
\put(5.5,9){5}

\put(9.5,5){\bf{1}}

\end{picture}

\end{center}

The $\bd$-quotient-partition of $\la$ is $\bar{q}_3(\la)= ( 8, \, 4, \, 2, \, 1)$. It has
the following shifted diagram, where we indicate alongside the rim the
runners to consider (i.e. the hand residue at the end of rows, and the
foot residue decreased by 1 at the end of columns):

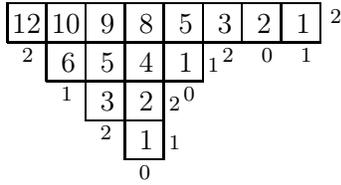
\setlength{\unitlength}{1.3mm}

\linethickness{0.2mm}

\begin{center}

\begin{picture}(40,22)

\put(0,20){\line(1,0){32}}
\put(0,16){\line(1,0){32}}

\put(4,12){\line(1,0){16}}
\put(8,8){\line(1,0){8}}
\put(12,4){\line(1,0){4}}

\put(0,20){\line(0,-1){4}}
\put(4,20){\line(0,-1){8}}
\put(8,20){\line(0,-1){12}}
\put(12,20){\line(0,-1){16}}
\put(16,20){\line(0,-1){16}}
\put(20,20){\line(0,-1){8}}
\put(24,20){\line(0,-1){4}}
\put(28,20){\line(0,-1){4}}
\put(32,20){\line(0,-1){4}}

\put(33,18){\scriptsize{$2$}}
\put(30,14){\scriptsize{$1$}}
\put(26,14){\scriptsize{$0$}}
\put(22,14){\scriptsize{$2$}}

\put(20.5,13){\scriptsize{$1$}}

\put(18,10){\scriptsize{$0$}}

\put(16.5,9){\scriptsize{$2$}}
\put(16.5,5){\scriptsize{$1$}}

\put(13.5,2){\scriptsize{$0$}}

\put(9.5,6){\scriptsize{$2$}}
\put(5.5,10){\scriptsize{$1$}}
\put(1.5,14){\scriptsize{$2$}}

\put(29.5,17){1}
\put(25.5,17){2}
\put(21.5,17){3}
\put(17.5,17){5}
\put(13.5,17){8}
\put(9.5,17){9}
\put(4.5,17){10}
\put(0.5,17){12}

\put(17.5,13){1}
\put(13.5,13){4}
\put(9.5,13){5}
\put(5.5,13){6}

\put(13.5,9){2}
\put(9.5,9){3}

\put(13.5,5){1}

\end{picture}

\end{center}

We can now compute the modified bar lengths. The normalized $\ba$-set $\{0,1, 3, 4 ,6 ,7, 8, 9 ,10 ,12, 13 ,14, 19, 25, 28 \}$ for $\Dla$ gives us $x_0=5$,
$x_1=8$ and $x_2=2$. For $0 \leq i , \,  j \leq 2$ and 
$z \in \cBij(\bar{q}_3(\la))$, we have $\ov{h}(z)=h(z)+3(x_i-x_j)$, 
and $\ov{\cB}_{i \rightarrow j}(\bar{q}_3(\la))=
\{ |\ov{h}(z)| \; | \; z \in \cBij(\bar{q}_3(\la)\}$. This gives the following:

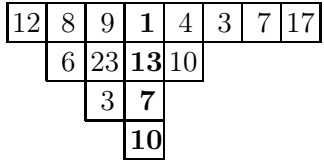
\setlength{\unitlength}{1.3mm}

\linethickness{0.2mm}

\begin{center}

\begin{picture}(40,22)

\put(0,20){\line(1,0){32}}
\put(0,16){\line(1,0){32}}

\put(4,12){\line(1,0){16}}
\put(8,8){\line(1,0){8}}
\put(12,4){\line(1,0){4}}

\put(0,20){\line(0,-1){4}}
\put(4,20){\line(0,-1){8}}
\put(8,20){\line(0,-1){12}}
\put(12,20){\line(0,-1){16}}
\put(16,20){\line(0,-1){16}}
\put(20,20){\line(0,-1){8}}
\put(24,20){\line(0,-1){4}}
\put(28,20){\line(0,-1){4}}
\put(32,20){\line(0,-1){4}}

\put(28.5,17){17}
\put(25.5,17){7}
\put(21.5,17){3}
\put(17.5,17){4}
\put(13.5,17){\bf{1}}
\put(9.5,17){9}
\put(5.5,17){8}
\put(0.5,17){12}

\put(16.5,13){10}
\put(12.5,13){\bf{13}}
\put(8.5,13){23}
\put(5.5,13){6}

\put(13.5,9){\bf{7}}
\put(9.5,9){3}

\put(12.5,5){\bf{10}}

\end{picture}

\end{center}

It is then easy to check that the result announced by Theorem \ref{main} does hold. We just explicitely describe the case of parts (indicated in bold in the above diagrams). We see that, in accordance with Corollary \ref{inclusions}, $\cP(\la)=\{ 13, \, 10 , \, 4 \} =\cP(\bar{c}_3(\la)) \circ \ov{\cP}(\bar{q}_3(\la))$. And, for the last four bars in bold in the diagram of $\la$, we have $\{1, \, 2, \, 7, \, 14\}=\{1, \, 7\} \cup \{2, \, 14\}=[\cP(\bar{c}_3(\la)) \cap \ov{\cP}(\bar{q}_3(\la))] \cup [\cP(\bar{c}_3(\la)) \cap \ov{\cP}(\bar{q}_3(\la))]^{\times 2}$.

\medskip

In \cite[Corollary 4.12]{BeGrOl}, a generalization of a relative hook
formula discovered by G. Malle and G. Navarro was presented. We finish
this paper by the bar analogue of \cite[Corollary 4.12]{BeGrOl}.

If $\la$ is a bar partition of $n$ we let again $\rho_{\la}$
be an irreducible spin character of  $\widehat{\sym_n}$ labelled by
$\la.$ We define $\sigma(\la)=|\la|-m(\la)$, and $\delta(\la)={\floor{ \sigma(\la)/2}}$ so that 
the bar formula reads 
$$\rho_{\la}(1)=2^{\delta(\la)}\frac{n!}{\pi {\cal B}(\la)}.$$
By Corollary \ref{inclusions} (3), we have that
$m(\la)=m(\lad)+m(\bdqla)-2 \delta$, where $\delta=|\cP(\lad) \cap
\ov{\cP}(\bdqla) |$. It follows from this and from
$|\la|=|\bdqla|+|\lad|$ 
that $\sa(\la)=\sa(\lad)+\sa(\bdqla)+2 \delta$. We thus have
\begin{equation}\label{deltaeq}
\delta(\la)=\delta(\bdqla) + \delta(\lad)+\delta + \varepsilon ,
\end{equation}
where $\varepsilon=0$ if $\sigma(\la)$ is odd or if $\sigma(\la)$ and
$\sigma(\lad)$ are both even, and  $\varepsilon=1$ otherwise. Theorem
\ref{main} now implies that 
$$\pi {\cal B}(\la) =\pi\widetilde{\cB}(\bdqla) . \pi {\cal B}(\lad)=2^{\delta}\pi\ov{\cB}(\bdqla) . \pi {\cal B}(\lad).$$
Combining this with formula (\ref{deltaeq}) we get a relative bar formula:

\begin{corollary} With the above notation
$$\rho_{\la}(1)=
\frac{|\la|! }{|\lad|!}\cdot \frac{2^{\delta(\bdqla)+\varepsilon}}{\pi\ov{\cB}(\bdqla)}\rho_{\lad}(1).$$

\end{corollary}

\noindent\textbf{Acknowledgements.}\quad
J.-B. Gramain gratefully acknowledges financial support from a grant
of the Agence Nationale de la Recherche (number
ANR-10-PDOC-021-01). He also wishes to express his gratitude to
J. B. Olsson and J. Grodal for their (not only financial) support
during his stay at the University of Copenhagen, where this work was
done. Finally, the authors thank C. Bessenrodt for useful discussions
and for a careful reading of the manuscript, thereby pointing 
out a problem in an earlier version of this work.

\end{document}